\documentclass[11pt, reqno]{amsart}

\usepackage{version}
\usepackage{amsmath,amsfonts,amssymb,amsthm}
\usepackage{mathtools}
\usepackage{mathrsfs}
\usepackage{nicefrac}
\usepackage[all]{xy}
\newtheorem{theorem}{Theorem}[section]
\newtheorem{lemma}[theorem]{Lemma}
\newtheorem{corollary}[theorem]{Corollary}
\newtheorem{proposition}[theorem]{Proposition}

\theoremstyle{definition}
\newtheorem{definition}[theorem]{Definition}

\newtheorem{remark}[theorem]{Remark}

\newcommand{\B}{\mathbb{B}}
\newcommand{\C}{\mathbb{C}}
\newcommand{\D}{\mathbb{D}}
\renewcommand{\H}{\mathbb{H}}
\newcommand{\N}{\mathbb{N}}

\newcommand{\R}{\mathbb{R}}

\newcommand{\Z}{\mathbb{Z}}

\newcommand{\id}{{\sf id}}

\def\de{\partial}

\def\v{\varphi}

\renewcommand{\Im}{{\sf Im}\,}

\numberwithin{equation}{section}

\begin{document}
\title[Simultaneous models]{Simultaneous models for commuting holomorphic self-maps of the ball}
\author[L. Arosio]{Leandro Arosio$^\dag$}
\author[F. Bracci]{Filippo Bracci$^{\dag\dag}$}
\address{L. Arosio, F. Bracci: Dipartimento Di Matematica\\
Universit\`{a} di Roma \textquotedblleft Tor Vergata\textquotedblright\ \\
Via Della Ricerca Scientifica 1, 00133 \\
Roma, Italy} \email{arosio@mat.uniroma2.it, fbracci@mat.uniroma2.it}
\thanks{$^{\dag}$Supported by the SIR grant ``NEWHOLITE - New methods in holomorphic iteration'' n. RBSI14CFME}
\thanks{$^{\dag\dag}$Partially supported by the ERC grant ``HEVO - Holomorphic Evolution Equations'' n. 277691}
\begin{abstract}
We prove that a finite family of commuting holomorphic self-maps of the unit ball $\B^q\subset \C^q$ admits a simultaneous holomorphic conjugacy to a family of commuting automorphisms of a possibly lower dimensional ball, and that such conjugacy satisfies a universal property.  As an application we describe  when a hyperbolic and a parabolic holomorphic self-map of $\B^q$ can commute.
\end{abstract}
\maketitle

\section{Introduction}

In this paper we study simultaneous linearization  of commuting holomorphic self-maps of the unit ball $\B^q\subset \C^q$, and we use  such a tool to find obstructions to commutation. 

In order to state the results, we need to  recall what we mean by linearization of a holomorphic self-map. Even if we are interested in the general case of a holomorphic self-map of $\B^q$, for the sake of clearness we will  consider at first the  case of  a univalent self-map of the unit disc $\D\subset \C$.
It follows from the work of Schr\"oder \cite{Sc}, Valiron \cite{va2}, Pommerenke \cite{Po}, Baker and Pommerenke \cite{BaPo} (see also Cowen \cite{Co1} and Bourdon-Shapiro \cite{BS}), that given a univalent map $f\colon \D\to \D$  there exists a conjugacy as in the following commutative diagram:
$$\xymatrix{\D\ar[r]^{f}\ar[d]_{h}& \D\ar[d]^{h}\\
\Omega\ar[r]^{\v}& \Omega,}$$ 
where $\Omega$ is either $\D$ or $\C$, $h\colon \D\to \Omega$ is a univalent map, and $\v$  is an automorphism of $\Omega$, and thus a linear fractional transformation. Let $k\in \N$, $k\geq 2$.
Given a $k$-uple of commuting univalent self-maps $(f_1,\dots, f_k)$ of $\D$,  it is then natural to ask whether  
such a conjugacy can be performed simultaneously. In other words, the question is whether there exists  a triple $(\Omega,h,\Phi)$, where
$\Omega$ is either $\D$ or $\C$, $h\colon \D\to \Omega$ is univalent, $\Phi\coloneqq (\varphi_1,\dots ,\varphi_k)$ is a $k$-uple of commuting automorphisms of $\Omega$ and
$$h\circ f_j=\varphi_j\circ h,\quad  1\leq j\leq k.$$ 
Applying the previous result to the univalent self map $f_1\circ \dots\circ f_k$,  Cowen answered positively this question in \cite[Theorem 3.1]{Co2}.

Let us now move to higher dimensions. Linearization  for holomorphic self-maps of the unit ball has been studied in particular cases and under more or less stringent regularity assumptions by many authors, see, {\sl e.g.}, \cite{Bay, Bay2, Bay3,  cinzia1, cinzia2, BrGe, CM}.
For commuting holomorphic self-maps with a common isolated fixed point in $\B^q$, the simultaneous linearization has been investigated by MacCluer in \cite{Mac} exploiting constructions done by Cowen--MacCluer in \cite{CM} (see also \cite{KRS} and  \cite{BiGe}).  The case of commuting holomorphic self-maps with no fixed points in $\B^q$  was studied by Gentili and the second named author \cite{BrGe} under suitable regularity assumption at the Denjoy--Wolff point, which  allowed linearization by  iterative methods.
However, without any additional assumption the  iterative methods seem to fail (see \cite{BFG}).

In this paper we propose a new approach to simultaneous linearization in several variables, generalizing the theory of canonical models introduced in \cite{AB,A} from the case of a single holomorphic self-map $f$ of $\B^q$ to the case of a $k$-uple of commuting holomorphic self maps $(f_1, \dots ,f_k)$ of $\B^q$. If $f\colon \B^q\to \B^q$ is a univalent self-map, then it is shown in \cite{AB} that performing a direct limit  one can always find a conjugacy as follows:
$$\xymatrix{\B^q\ar[r]^{f}\ar[d]_{h}& \B^q\ar[d]^{h}\\
\Omega\ar[r]^{\psi}& \Omega,}$$ where $\Omega$ is a  $q$-dimensional complex manifold uniquely determined up to biholomorphism, $h\colon \B^q\to \Omega$ is a univalent map and $\psi$ is an automorphism of $\Omega$. 
The triple $(\Omega, h,\psi)$ is called a {\sl model} for $f$.
Due to the lack of a uniformization theorem in higher dimensions, the complex structure of $\Omega$ is not known in general.  However, identifying points of $\Omega$ which are at zero Kobayashi-pseudodistance, one  obtains a 
conjugacy as follows:
\begin{equation}\label{diagram}
\xymatrix{\B^q\ar[r]^{f}\ar[d]_{\ell}& \B^q\ar[d]^{\ell}\\
\B^d\ar[r]^{\tau}& \B^d,}
\end{equation}
where $0\leq d\leq q$, $\ell\colon \B^q\to \B^d$ is a holomorphic map and  $\tau$ is an automorphism of $\B^d$. The triple $(\B^d,\ell, \tau)$ in \eqref{diagram} is called the {\sl canonical Kobayashi hyperbolic semi-model} for $f$, 
and is characterized uniquely up to biholomorphisms by the following  universal property: every other conjugacy of $f$ with an automorphism of a Kobayashi hyperbolic complex manifold factorizes through the triple $(\B^d,\ell,\tau)$.
In particular, one has $d=0$ if and only if  the only possible conjugacy of $f$  with an automorphism of a Kobayashi hyperbolic complex manifold  is given by the trivial conjugacy to a point. We call the number $d$  the {\sl type} of $f$. 

Recall that a holomorphic self-map of $\B^q$ is said {\sl elliptic} if it has a fixed point in $\B^q$. A non-elliptic holomorphic self-map $f$ of $\B^q$ has a unique fixed (with respect to $K$-limits) boundary point $p\in\partial\B^q$ such that the dilation of $f$ at $p$ is less than or equal to $1$,  called the {\sl Denjoy--Wolff point} of $f$. A non-elliptic map $f$ is called {\sl hyperbolic} if  the dilation at its Denjoy--Wolff point is strictly less than $1$, and {\sl parabolic} if the dilation is equal to $1$  (see Section \ref{ball} for precise definitions).   

If $f$ is  hyperbolic, then the type $d$ is always strictly greater than $0$, and the  automorphism $\tau$ is  hyperbolic  with the same dilation at its Denjoy--Wolff point.

While the model $(\Omega, h, \psi)$ does not exist in general for holomorphic self-maps of $\B^q$ which are not univalent,  it is shown in \cite{A} that the canonical Kobayashi hyperbolic semi-model  $(\B^d, \ell, \tau)$ always exists. 

Our first result shows that the concept of canonical Kobayashi hyperbolic semi-model can be generalized to a  $k$-uple of commuting holomorphic self-maps $(f_1,\dots, f_k)$ of $\B^q$.

\begin{theorem}\label{thm-intro1}
Let $(f_1,\ldots, f_k)$ be a $k$-uple of commuting holomorphic self-maps of $\B^q$. Then there exist $0\leq d\leq q$, a holomorphic map $\ell:\B^q\to \B^d$ and a $k$-uple of commuting automorphisms  $(\tau_1,\ldots, \tau_k)$ of $\B^d$ such that 
\[
\ell \circ f_j=\tau_j\circ \ell, \quad 1\leq j\leq k.
\]


Moreover, if $\Lambda$ is a Kobayashi hyperbolic complex manifold such that there exists a holomorphic map $g:\B^q\to \Lambda$ and a $k$-uple of commuting automorphisms $(\varphi_1,\ldots, \varphi_k)$ of $\Lambda$ such that $$g \circ f_j=\varphi_j\circ g, \quad  1\leq j\leq k,$$ then there exists a unique holomorphic mapping $\eta:\B^d\to \Lambda$  such that $g=\eta \circ \ell$ and $\eta\circ \tau_j=\varphi_j \circ \eta$ for all $1\leq j\leq k$.
\end{theorem}
The triple $(\B^d,\ell,(\tau_1,\ldots, \tau_k) )$ is called the   {\sl canonical Kobayashi hyperbolic semi-model} for $(f_1,\ldots, f_k)$.
Theorem \ref{thm-intro1} actually holds in  the more general case of a cocompact Kobayashi hyperbolic complex manifold, {\sl e.g.}, every bounded homogenous domain of $\C^q$ (see Theorem \ref{principaleforward}). 

Something more can be said about $(\B^d,\ell,(\tau_1,\ldots, \tau_k) )$.
It follows from Proposition \ref{disdiv} that if there is $1\leq j\leq k$ such that $f_j$ is hyperbolic, then the dimension $d$ is strictly greater than zero, and the  automorphism $\tau_j$ is  hyperbolic  with  the same dilation  at its Denjoy--Wolff point.
It is also important to notice that for a fixed $1\leq j\leq k$  the triple $(\B^d, \ell, \tau_j)$ is in general not  the canonical Kobayashi hyperbolic semi-model for  $f_j$. In fact $d\leq   {\rm type}(f_j)$, and  strict inequality can occur.
On the other hand, the triple $(\B^d,\ell, \tau_1\circ  \dots \circ \tau_k)$ is the canonical Kobayashi hyperbolic semi-model for  $f_1\circ \dots \circ f_k$. This shows that the dimension $d$ is  related to the dynamics of the map $ f_1\circ \dots \circ f_k$.

We now turn to the following problem: if $f,g\colon \B^q\to \B^q$ are two commuting holomorphic self-maps, what can be said about their ellipticity/parabolicity/hyperbolicity? 
If one of the maps is elliptic, then the answer is trivial.
Indeed, if $f$ has a fixed point $z\in \B^q$, then for all $n\geq 0$ the point $g^n(z)$ is fixed by $f$. Hence, if $f$ has a unique fixed point, then  $g$ has to be elliptic too.
On the other hand,  it is easy to find examples of elliptic mappings  (with affine slices of fixed points) commuting with parabolic or hyperbolic mappings (see \cite{Br}). Hence in the following discussion we will assume that $f$ and $g$ are not elliptic.

Heins \cite{He} proved that if $q=1$ and $f$ is a hyperbolic automorphism of $\D$, then $g$ is a hyperbolic automorphism as well.
Behan \cite[Thm. 6]{Behan} and Shields \cite[Thm. 1]{Shields} proved that if $q=1$ and  $f$ is not a hyperbolic automorphism of $\D$, 
then $f$ and $g$ have the same Denjoy--Wolff point.
As a consequence of these results, in  dimension one $f$ and $g$ may have different Denjoy--Wolff point only if they are
both hyperbolic automorphisms. Cowen \cite[Cor. 4.1]{Co2}, starting from the previous results and the existence of models, proved the following result.
\begin{theorem}[Cowen]\label{Cowen-intro}
Let $f, g$ be two nonelliptic commuting holomorphic self-maps of $\D$. 
If $f$ is hyperbolic, then $g$ is hyperbolic (and thus if $g$ is parabolic, then $f$ is parabolic).
\end{theorem}

In higher dimensions, it was proved in \cite{Br} that $f$ and $g$  have the same Denjoy--Wolff point unless both are hyperbolic automorphisms when restricted to the one-dimensional slice connecting the two  Denjoy--Wolff points. 
In \cite{GdF} de Fabritiis--Gentili   proved that if $f$ is a hyperbolic automorphism of $\B^q$, then $g$ is hyperbolic and it is a hyperbolic automorphism when restricted to the one-dimensional slice connecting the two fixed points of $f$. Later on,
de Fabritiis  \cite{deFa} completely characterized the mapping $g$ in this case. Bayart \cite[Thm. 6.1]{Bay}  proved, under regularity conditions at the Denjoy--Wolff point, that if $f$ is a hyperbolic self-map of $\B^q$, then it cannot commute with certain parabolic maps.

Using the theory developed in the present paper, we are able to describe when a hyperbolic and a parabolic self-map of $\B^q$ can commute   in terms of their types   (see Theorem \ref{teo-aut}, Proposition \ref{para-0}, Proposition \ref{para-1}). 
Recall that the {\sl step}  of $g$ at $x_0$ is defined as the limit $s_1^g(x_0)\coloneqq \lim_{n\to\infty} k_{\B^q}(g^{n+1}(x_0),g^n(x_0))$, where $k_{\B^q}$ denotes the Kobayashi distance. 
\begin{theorem}\label{thm-intro2}
Let $f, g$ be two nonelliptic commuting holomorphic self-maps of $\B^q$. 
\begin{enumerate}
\item If $f$ is hyperbolic of  type $q$, then $g$ is hyperbolic.
\item If $g$ is parabolic of type $0$, then $f$ is parabolic.
\item If $g$ is parabolic of type $1$ and there exists $x_0$ such that $s_1^g(x_0)>0$, then $f$ is parabolic.

\end{enumerate}
Moreover, there exist  counterexamples in all remaining cases, that is,
for all $u,v\in \N$ such that 
$$1\leq u\leq q-1,\quad 2\leq v\leq q,$$
there exist  commuting holomorphic self-maps $f, g\colon \B^q\to \B^q$ such that $f$ is hyperbolic of type $u$ and  $g$ is parabolic of type $v$.
\end{theorem}

The assumption on the step in (3) could be redundant, indeed it is actually an open question whether there exists a parabolic self map $g\colon \B^q\to \B^q$ with  ${\rm type}(g)\geq 1$ which  admits a point $x_0$ such that $s_1^g(x_0)=0$. 
 
Notice that Theorem \ref{thm-intro2} (1) generalizes the result of  de Fabritiis--Gentili \cite{GdF}, and generalizes Cowen's Theorem \ref{Cowen-intro} to higher dimensions, since, if $q=1$, then ${\rm type}(f)=1=q$. The proof of Theorem \ref{thm-intro2} (1)   is based on the following result  which is of interest by itself and is proved in Proposition \ref{univ} in the more general case of a $k$-uple of commuting holomorphic self-maps.
\begin{proposition}
 Let $f\colon \B^q\to \B^q$ be a holomorphic map with  ${\rm type}(f)=q$. Let $(\B^q,\ell,\tau)$ be a canonical Kobayashi hyperbolic semi-model for $f$. Then for every $R>0$ and every $z\in \B^q$, both $\ell$ and $f$ are   univalent on the Kobayashi ball $B(f^n(z),R)$ for $n$ large enough. 
\end{proposition}

\medskip

The authors thank the anonymous referees for their comments which improved  the paper.

\section{Semi-models for commuting maps}

Let $X$ be a complex manifold. We denote by $k_X$ the Kobayashi pseudodistance and by $\kappa_X$ the Kobayashi pseudometric.
\begin{definition}
We say that $X$ is {\sl cocompact} if $X/{\rm aut}(X)$ is compact. 
\end{definition}
Notice that this implies that $X$ is complete Kobayashi  hyperbolic \cite[Lemma 2.1]{FS}.

\begin{definition}
Let $k\in \N$, $k\geq 1$. Let $A$ a set. We call a {\sl $k$-sequence} in $A$ any  mapping $a\colon \N^k\to A$, and we denote it by $(a_N)_{N\in\N^k}$.
We consider the following partial ordering on $\N^k$: $$(m_1,\dots , m_k)\geq (n_1,\dots, n_k)$$ if for all $1\leq j\leq k$ we have $m_j\geq n_j.$ 
We denote  by $E_j\in \N^k$ the $k$-uple  with 1 as $j$-th entry and 0 everywhere else. 

Let $N=(n_1,\dots, n_k)\in \N^k$. If $F=(f_1,\dots, f_k)$ is a $k$-uple of commuting holomorphic self-maps of a complex manifold $X$, we denote by $F^N$ the holomorphic self map of $X$ defined by $f_1^{n_1}\circ \dots \circ f_k^{n_k}$. 

Let $F=(f_1,\dots, f_k)$ be a $k$-uple of commuting holomorphic self-maps of   $X$, and let  $G=(g_1,\dots, g_p)$ be a $p$-uple of commuting holomorphic self-maps of $X$. We say that $F$ and $G$ commute if $$f_j\circ g_i=g_i\circ f_j,\quad\forall \,1\leq j\leq k,1\leq i\leq p.$$

\end{definition}
\begin{remark}
Notice that $F^{E_j}=f_j$ for all $1\leq j\leq k$.

\end{remark}

\begin{remark}
We consider the  base of neighborhoods for the point $\infty$ in $\N^k\cup\{\infty\}$ given by $\{N\geq M\}_{M\in \N^k}.$
Thus, if  $(a_N)_{N\in \N^k}$ is a $k$-sequence in a Hausdorff space $Y$, by  
$\lim_{N\to\infty} a_N=\ell\in Y$ we mean that for any neighborhood $W$ of $\ell$ in  $Y$, there exists $N_0\in \N^k$ such that $a_N\in W$ for all $N\geq N_0$.
\end{remark}
The following lemmas are classical and the proof is omitted.
\begin{lemma}
Let $(a_N)_{N\in \N^k}$ be a $k$-sequence in a Hausdorff space $Y$. Then $\lim_{N\to\infty} a_N=\ell$ if and only if
for every   sequence $(N_h)_{h\in \N}$ in $\N^k$ converging to $\infty$ 
we have $\lim_{h\to\infty} a_{N_h}=\ell.$
\end{lemma}
\begin{lemma}
 Let  $(a_N)_{N\in \N^k}$ be a monotone non-increasing $k$-sequence in $[0,+\infty)$. Then 
$$\lim_{N\to\infty} a_N=\inf_{N\in \N^k} a_N.$$
\end{lemma}

In what follows, let $k\in \N, k\geq 1$.

\begin{definition}
Let $X$ be a complex manifold.
We call   {\sl (forward)  holomorphic $k$-dynamical system} on $X$ any
family $(F_{N,M}\colon X\to X)_{M\geq N\in \N^k}$ of holomorphic self-maps such that for all $M\geq U\geq  N\in \N^k$, we have
$$F_{U,M}\circ F_{N,U}= F_{N,M}.$$ 

Let $F\coloneqq(f_1,\dots, f_k)$  be a $k$-uple of commuting holomorphic self-maps $f_j\colon X\to X$. Then setting $F_{N,M}\coloneqq F^{M-N}$ for all $M\geq N\in \N^k$, we obtain a   holomorphic $k$-dynamical system naturally associated with $F$.
\end{definition}

\begin{definition}\label{directdef}
Let $X$ be a complex manifold and let $(F_{N,M}\colon X\to X)_{M\geq N\in \N^k}$ be a   holomorphic $k$-dynamical system.
We call {\sl canonical Kobayashi hyperbolic direct limit} (CDL for short) for $(F_{N,M})$ a pair $(Z,\alpha_N)$ where $Z$ is a Kobayashi hyperbolic  complex manifold and
$(\alpha_N\colon X\to Z)_{N\in \N^k}$ is a family of holomorphic mappings such that 
 $$\alpha_M\circ F_{N,M}=\alpha_N,\quad \forall \ M\geq N\in \N^k,$$
 which satisfies the following universal property:
if $Q$ is a Kobayashi hyperbolic complex manifold and if $(g_N\colon X\to Q)_{N\in \N^k}$ is a family of holomorphic mappings satisfying
$$g_M\circ F_{N,M}=g_N,\quad \forall \ M\geq N\in\N^k,$$
then there exists a  unique holomorphic mapping $\Gamma\colon Z\to Q$ such that
$$g_N=\Gamma\circ \alpha_N,\quad \forall\ N\in \N^k.$$
\end{definition}

The proof of the following lemma is trivial and is therefore omitted.
\begin{lemma}
The canonical Kobayashi hyperbolic direct limit of $(F_{N,M})$ is essentially unique, in the following sense.
Let $(Z,\alpha_N)$ and $(Q, g_N)$ be two CDL  for $(F_{N,M})$. Then there exists a biholomorphism 
$\Gamma\colon Z\to Q$ such that $$g_N=\Gamma\circ \alpha_N,\quad \forall\ N\in \N^k.$$
\end{lemma}
\begin{definition}
Let  $X,Y$ be complex manifolds, and  let $f\colon X\to Y$ be a holomorphic map. Then by $f^*k_Y$ and   $f^*\kappa_Y$ we denote the pullbacks 
$$f^*k_Y(x,y):=k_Y(f(x),f(y)),\quad x,y\in X,$$
$$ f^*\kappa_Y(x,v):=\kappa_Y(f(z), d_xf(v)), \quad x\in X, v\in T_xX.$$
\end{definition}

\begin{lemma}\label{directlimit}
Let $X$ be a cocompact  Kobayashi hyperbolic complex manifold, and  let $(F_{N,M}\colon X\to X)_{M\geq N\in \N^k}$ be a   holomorphic $k$-dynamical system. There exists a CDL $(Z,\alpha_N)$ for $(F_{N,M})$, where 
$Z$ is a holomorphic retract of $X$. Moreover,
 \begin{equation}\label{eqdir2}
 Z=\bigcup_{N\in \N^k}\alpha_N(X),
 \end{equation}
\begin{equation}\label{eqdir3}
\lim_{M\to\infty}  (F_{N,M})^*\, k_X=\alpha_N^* \, k_Z,\quad
\lim_{M\to\infty} (F_{N,M})^*\, \kappa_X=\alpha_N^* \, \kappa_Z,\quad  \forall \, N\in \N^k.
\end{equation} 
\end{lemma}
\begin{proof}
Let $(N_h)_{h\in \N}$ be a strictly monotone sequence in $\N^k$ converging to $\infty$.  
Define a  holomorphic dynamical system $(g_{h,j}\colon X\to X)_{j\geq h\in \N}$ by $$g_{h,j}\coloneqq F_{N_h,N_j}.$$
Let $(Z, \alpha_h)$ be the CDL given by  \cite[Theorem 2.11]{A}  for $(g_{h,j})$. For every $N\in \N^k$ let $h\in \N$ be  such that  $N_h\geq N$, and define $\alpha_N\colon X\to Z$ by 
\begin{equation}\label{catilina}
\alpha_N\coloneqq \alpha_{h}\circ F_{N,N_h}.
\end{equation} It is easy to see that this is well defined and that $(Z, \alpha_N)$ is a CDL for $(F_{N,M})$ satisfying (\ref{eqdir2}). 

We are left to prove (\ref{eqdir3}).
Let $x,y\in X$,  let $N\in \N^k$ and let $h\in \N$ be  such that  $N_h\geq N$. Notice that the $k$-sequence $(k_X(F_{N,M}(x),F_{N,M}(y)))_{M\geq N\in \N^k}$ is monotone non-increasing, and thus it admits a limit as $M\to \infty$.
Then, using  \cite[Theorem 2.11 (2.8)]{A},
\begin{align*}
\lim_{M\to\infty}  k_X(F_{N,M}(x),F_{N,M}(y))&=\lim_{j\to\infty}  k_X(F_{N,N_j}(x),F_{N,N_j}(y))
=\\
&=\lim_{j\to\infty} k_X(g_{h,j}(F_{N,N_h}(x)),g_{h,j}(F_{N,N_h}(y)))=\\ &= k_Z(\alpha_h(F_{N,N_h}(x)),\alpha_h(F_{N,N_h}(y)))= k_Z(\alpha_N(x),\alpha_N(y)).
\end{align*}
The proof for the Kobayashi metric $\kappa_X$ is similar.

\end{proof}

\begin{definition}
Let $X$ be a complex manifold and let $F\coloneqq(f_1,\dots, f_k)$ be a $k$-uple of commuting holomorphic self-maps of $ X$. A {\sl  quasi-model}  for $F$ is a triple $(\Lambda,h,\Phi)$ where $\Lambda$ is a complex manifold, $h\colon X\to \Lambda$ is a holomorphic map, and $\Phi\coloneqq(\v_1,\dots, \v_k)$ is a $k$-uple of commuting automorphisms of $\Lambda$ 
such that $$h\circ f_j=\v_j\circ h,\quad \forall\,1\leq j\leq k.$$
 Let $(\Lambda,h,\Phi)$  and $(Z,\ell,T\coloneqq(\tau_1,\dots,\tau_k))$ be two quasi-models for $F$. A {\sl morphism of  quasi-models} $\hat\eta\colon (Z,\ell,T)\to(\Lambda,h,\Phi)$ is given by a holomorphic map $\eta\colon Z\to \Lambda$ such that 
 $\eta\circ \ell=h$ and such that 
 $$\eta\circ \tau_j=\v_j\circ \eta,\quad \forall\,1\leq j\leq k.$$
If the map $\eta \colon Z\to \Lambda$ is a biholomorphism, then we say that $\hat\eta\colon  (Z,\ell,T)\to (\Lambda, h, \Phi)$ is an {\sl isomorphism of  quasi-models}. Notice that then $\eta^{-1}\colon \Lambda\to Z$ induces a morphism $ {\hat\eta}^{-1}\colon  (\Lambda, h, \Phi)\to  (Z,\ell,T). $
\end{definition}

\begin{definition}
Let $X$ be a complex manifold and let $F\coloneqq(f_1,\dots, f_k)$ be a $k$-uple of commuting holomorphic self-maps of $X$. A {\sl  semi-model}  for $F$ is a  quasi-model $(\Lambda,h,\Phi)$ such that 
 \begin{equation}\label{due}
 \bigcup_{N\in \N^k}\Phi^{- N}(h(X))=\Lambda.
 \end{equation}
 If $(Z,\ell,T)$ and $(\Lambda,h,\Phi)$ are two semi-models for $F$, then a  morphism  of  quasi-models
 $\hat\eta\colon  (Z,\ell,T)\to (\Lambda, h, \Phi)$ is called  a {\sl morphism of  semi-models}.
\end{definition}

\begin{proposition}\label{help}
Let $X$ be a complex manifold and let $F\coloneqq(f_1,\dots, f_k)$ be a $k$-uple of commuting holomorphic self-maps of $X$. Let $(Z,\ell,T)$ be a  semi-model for $F$, and let  $(\Lambda,h,\Phi)$ be a  quasi-model for $F$. Then 
\begin{enumerate}
\item there exists at most one morphism from $ (Z,\ell,T)$ to $(\Lambda, h, \Phi)$, and
\item if $\hat\eta\colon (Z,\ell,T)\to (\Lambda, h, \Phi)$ is a morphism, then $$\eta(Z)=\bigcup_{N\in \N^k}\Phi^{- N}(h(X)).$$

\end{enumerate}
\end{proposition}
\begin{proof}
$\,$
\begin{enumerate}
\item Let $\hat \eta,\hat \beta \colon (Z,\ell,T)\to (\Lambda, h, \Phi)$ be two morphisms. Then for all $N\in \N^k$, $$\eta\circ T^{-N}\circ \ell=\Phi^{-N}\circ \eta\circ \ell=\Phi^{-N}\circ h=\Phi^{-N}\circ \beta\circ \ell=\beta\circ T^{-N}\circ \ell.$$
Since $(Z,\ell,T)$ is a semi-model, $\bigcup_{N\in \N^k} T^{-N}(\ell (X))=Z$, hence $\eta=\beta$.
\item $$\bigcup_{N\in \N^k}\Phi^{- N}(h(X))=\bigcup_{N\in \N^k}\Phi^{-N}(\eta(\ell(X)))=\eta\left(\bigcup_{N\in \N^k} T^{-N}(\ell(X))\right).$$
\end{enumerate}
\end{proof}

\begin{definition} 
Let $X$ be a complex manifold and let  $F\coloneqq(f_1,\dots, f_k)$ be a $k$-uple of commuting holomorphic self-maps of $X$. Let $(Z, \ell,T)$ be a semi-model for $F$  whose base space $Z$ is Kobayashi hyperbolic. We say that  $(Z, \ell,T)$ is a {\sl  canonical Kobayashi hyperbolic  semi-model} for  $F$ (for short  CSM)  if  it satisfies the following universal property: for any  semi-model
$(\Lambda, h,\Phi )$ for $F$ such that the base space $\Lambda$ is Kobayashi hyperbolic  there exists a  morphism of  semi-models $\hat\eta\colon (Z, \ell,T)\to (\Lambda, h,\Phi )$ (which is necessarily unique by Proposition \ref{help}).
 \end{definition}
 
 \begin{remark}\label{tango}
If $(Z, \ell,T)$ and $(\Lambda, h,\Phi )$ are two  CSM for  $F$, then they are isomorphic.
\end{remark}

\begin{theorem}\label{principaleforward}
 Let $X$ be a cocompact  Kobayashi hyperbolic complex manifold, and  let  $F\coloneqq(f_1,\dots, f_k)$ be a $k$-uple of commuting holomorphic self-maps of $X$. Then there exists a   CSM $(Z,\ell,T)$ for $F$, where $Z$ is a holomorphic retract of $X$.
Moreover, for all $N\in \N^k$,
\begin{equation}\label{kobdyn}
\lim_{ M \to\infty}  (F^M)^*\, k_X=(T^{- N}\circ \ell)^* \, k_Z,\quad
\lim_{M \to\infty}  (F^M)^*\, \kappa_X=(T^{- N} \circ\ell)^* \, \kappa_Z.
\end{equation}
\end{theorem}
\begin{proof}
Let $(Z,\alpha_N)$ be the CDL given by Lemma \ref{directlimit} for the   holomorphic $k$-dynamical system associated with $F$.
Let $1\leq j\leq k$. The $k$-sequence of holomorphic maps $(\beta_N\coloneqq\alpha_N\circ F^{E_j}\colon X\to Z)_{N\in \N^k}$ satisfies, for all $M\geq N\in \N^k$,
$$\beta_M\circ F_{M-N}=\alpha_M\circ F^{E_j}\circ F^{M-N}=\alpha_M\circ F^{M-N}\circ F^{E_j}=\alpha_N\circ F^{E_j}=\beta_N.$$
By the universal property in Definition \ref{directdef} there exists a unique holomorphic self-map $\tau_j\colon Z\to Z$ such that for all $N\in \N^k$, 
 $$\tau_j\circ \alpha_N=\alpha_N\circ F^{E_j}.$$
We claim that $\tau_j$ is a holomorphic automorphism. For all $N\in \N^k$, set $\gamma_N\coloneqq \alpha_{N+{E_j}}$. For all $M\geq N\in \N^k$,
$$\gamma_M\circ F_{M-N}=\alpha_{M+{E_j}}\circ F^{M-N}=\alpha_{N+{E_j}}=\gamma_N.$$
Thus there exists a holomorphic self-map $\delta_j\colon Z\to Z$ such that $\delta_j\circ \alpha_N=\alpha_{N+{E_j}}$ for all $N\in \N^k$.
For all $N\in \N^k$ we have $$\tau_j\circ\delta_j\circ \alpha_N=\tau_j\circ \alpha_{N+{E_j}}=\alpha_N,$$ and
$$\delta_j\circ \tau_j\circ \alpha_N=\delta_j\circ \alpha_N\circ F^{E_j}=\alpha_{N+{E_j}}\circ F^{E_j}=\alpha_N.$$ By the universal property in Definition \ref{directdef} we have that $\tau_j$ is a holomorphic automorphism and $\delta_j=\tau_j^{-1}$. 
Set $T\coloneqq (\tau_1,\dots,\tau_k)$. Since for all $N\in \N^k$, $$T^N\circ \alpha_N=\alpha_N\circ F^N=\alpha_0,$$ it follows that $\alpha_N=T^{-N}\circ \alpha_0$.

Set $\ell\coloneqq \alpha_0$. We claim that the triple $(Z,\ell, T)$ is a  CSM for $F$. 
Indeed, let $(\Lambda, h,\Phi )$ be a  semi-model for $F$ such that the base space $\Lambda$ is Kobayashi hyperbolic. For all $N\in \N^k$, let $\lambda_N\coloneqq  \Phi^{-N}\circ h$. Then by  the universal property in Definition \ref{directdef} there exists a unique holomorphic map $\eta\colon Z\to \Lambda$ such that for all $N\in \N^k$ we have $\eta\circ \alpha_N= \lambda_N,$
that is $$\eta\circ T^{-N}\circ \ell= \Phi^{-N}\circ h.$$
Setting $N=0$ this implies  $\eta\circ \ell=h$.
Now fix $1\leq j\leq k$.
If $N\in \N^k$, $$\Phi^{E_j}\circ\eta\circ T^{-{E_j}}\circ \alpha_N=\Phi^{E_j}\circ\eta\circ T^{-{E_j}}\circ T^{-N}\circ \ell=\Phi^{E_j}\circ \Phi^{-N-{E_j}}\circ h=\lambda_N.$$
By the universal property in Definition \ref{directdef}, we have $$\eta=\Phi^{E_j}\circ\eta\circ T^{-{E_j}}.$$
 Hence the map $\eta\colon Z\to \Lambda$ gives a morphism of semi-models $\hat\eta\colon  (Z,\ell,T)\to(\Lambda, h,\Phi )$.
\end{proof}

\begin{remark}\label{defgiusta}
 Let $X$ be a cocompact  Kobayashi hyperbolic complex manifold, and  let  $F\coloneqq(f_1,\dots, f_k)$ be a $k$-uple of commuting holomorphic self-maps of $X$. Let $(Z,\ell,T)$ be  a CSM for $F$.
The proof of Theorem \ref{principaleforward} actually shows that $(Z,\ell,T)$ satisfies the following slightly stronger universal property:
for any  quasi-model
$(\Lambda, h,\Phi )$ for $F$ such that the base space $\Lambda$ is Kobayashi hyperbolic,  there exists a  morphism of  quasi-models $\hat\eta\colon (Z, \ell,T)\to (\Lambda, h,\Phi )$ (which is necessarily unique by Proposition \ref{help}).
\end{remark}

\begin{definition}
Let $X$ be a  complex manifold, and  let  $F\coloneqq(f_1,\dots, f_k)$ be a $k$-uple of commuting holomorphic self-maps of $X$. Let $x\in X$, and let $M\in \N^k$. We define the {\sl M-step} of $F$ at $x$ as  $$s^F_M(x)\coloneqq \lim_{N\to\infty}k_X(F^N(x),F^{N+M}(x)).$$
Notice that the limit exists and $s^F_M(x)\in [0,+\infty)$ since $k_X(F^N(x),F^{N+M}(x))_{N\in\N^k}$ is a monotone non increasing $k$-sequence.
 \end{definition}

\begin{definition}
Let $X$ be a complex manifold. Let $f\colon X\to X$ be a holomorphic self-map.
The {\sl divergence rate} $c(f)\in [0,+\infty)$ of $f$ is defined as
$$c(f)\coloneqq \lim_{m\to\infty}\frac{k_{X}(f^m(x),x)}{m},$$ where $x\in X$.
It is shown in \cite{AB} that such a limit exists and does not depend on $x\in X$, and moreover
 $$c(f)=\inf_{m\in \N}\frac{k_X(f^m(x),x)}{m}=\lim_{m\to\infty}\frac{s^f_m(x)}{m}=\inf_{m\in \N}\frac{s^f_m(x)}{m}.$$
\end{definition}

\begin{lemma}\label{divcomm}
Let $X$ be a cocompact  Kobayashi hyperbolic complex manifold, and  let  $F\coloneqq(f_1,\dots, f_k)$ be a $k$-uple of commuting holomorphic self-maps of $X$. Let $(Z,\ell,T)$ be a CSM for $F$, and let $M\in \N^k$. Then for all $x\in X$, $$k_Z(\ell(x), T^M(\ell(x)))=s^F_M(x).$$
In particular, for all $1\leq j\leq k$, 
\begin{equation}\label{partdivtau}
c(\tau_j)=\lim_{m\to\infty} \frac{s^F_{mE_j}(x)}{m}=\inf_{m\to\infty} \frac{s^F_{mE_j}(x)}{m} .
\end{equation}
\end{lemma}
\begin{proof}
By (\ref{kobdyn}),

$$k_Z(\ell(x), T^M(\ell(x)))=k_Z(\ell(x),\ell(F^M(x))=\lim_{U\to\infty}k_X(F^U(x), F^{U+M}(x))=s^F_M(x).$$

\end{proof}

\begin{remark}\label{giants}
Notice that in general, if  $1\leq j\leq k$, 
$$s^F_{mE_j}(x)=\lim_{\N^{k-1}\ni Q\to\infty}s_m^{f_j}(\widehat F_j^Q(x))\leq s_m^{f_j}(x),$$
where $\widehat F_j$ denotes the $(k-1)$-uple $(f_1,\dots,f_{j-1}, f_{j+1},\dots,f_k)$.
\end{remark}

\begin{proposition}\label{disdiv}
Let $X$ be a cocompact  Kobayashi hyperbolic complex manifold, and  let  $F\coloneqq(f_1,\dots, f_k)$ be a $k$-uple of commuting holomorphic self-maps of $X$. Let $(Z,\ell,T)$ be a CSM for $F$. Then
for all $1\leq j\leq k$, $$c(\tau_j)=c(f_j).$$
\end{proposition}
\begin{proof}
Clearly $c(\tau_j)\leq c(f_j).$
For all $y\in X$ we have that $c(f_j)=\inf_{m\in \N}\frac{s^{f_j}_m(y)}{m},$  in particular, for all $m\in \N$ we have  $\frac{s^{f_j}_m(y)}{m}\geq c(f_j).$
Thus, by Remark \ref{giants}, $$\frac{s^F_{mE_j}(x)}{m}=\lim_{\N^{k-1}\ni Q\to\infty}\frac{s_m^{f_j}(\widehat F_j^Q(x))}{m}\geq c(f_j),$$
and the result follows then from (\ref{partdivtau}).
\end{proof}

\begin{definition}
Let $X$ be  a cocompact Kobayashi hyperbolic complex manifold. 
Let  $F\coloneqq(f_1,\dots, f_k)$ be a $k$-uple of commuting holomorphic self-maps of $X$.
The {\sl type} of $F$ (denoted ${\rm type}(F)$) is the dimension of a CSM for $F$. 
Clearly ${\rm type}(F)\leq {\rm dim }(X)$.
If ${\rm type}(F)={\rm dim }(X)$, then we say that $F$ is of {\sl automorphism type}. 
\end{definition}

Proposition \ref{disdiv} immediately yields the following corollary.
\begin{corollary}\label{dimensionemaggioreuno}
Let $X$ be a cocompact  Kobayashi hyperbolic complex manifold, and  let  $F\coloneqq(f_1,\dots, f_k)$ be a $k$-uple of commuting holomorphic self-maps of $X$ and assume that at least one of the $f_j$'s is hyperbolic.
Then ${\rm type}(F)\geq 1$.
\end{corollary}

\begin{definition}
Let $X$ be a Kobayashi hyperbolic complex manifold, let $p\in X$ and let $r>0$.
We denote by $B(p,r)$ the  ball of center $p$ and radius $r$ with respect to the Kobayashi distance $k_X$, in other words, $B(p,r):=\{z\in X: k_X(z,p)<r\}$.
\end{definition}

\begin{proposition}\label{iltipo}
 Let $X$ be a cocompact  Kobayashi hyperbolic complex manifold, and  let $F\coloneqq(f_1,\dots, f_k)$ be a $k$-uple of commuting holomorphic self-maps of $X$. Let $(Z,\ell,T)$ be a CSM for $F$, and let $d={\rm type}(F)$ be the dimension of $Z$. Then for all  $p\in X$ there exists $r>0$ and  $N_0\in \N^k$ such that for all $ N\geq N_0\in \N^k$
 \begin{enumerate}
 \item  ${\rm rk}_y\ell=d, \quad \forall \,y\in B(F^N(p),r)$,
  \item ${\rm rk}_yF^M\geq d, \quad \forall \,M\in \N^k, y\in B(F^N(p),r).$
  \end{enumerate}
\end{proposition}

\begin{proof}
We first prove (1) by contradiction. Fix $p\in X$. 
Assume that there exists a strictly increasing sequence $(N_h)_{h\in \N}$ in $\N^k$ with $N_0=0$ converging to $\infty$ such that for all $h\geq 1$ the mapping $\ell$ has not rank identically $d$ on  $B(F^{N_h}(p),\frac{1}{h})$.

Let $K$ be a compact subset in $X$ such that $X={\rm aut}(X)\cdot K$. 
 For all $N\in \N$, let  $\gamma_N$ be an automorphism of $X$ such that $\gamma_N(F^{N}(p))\in K$. 
Define a  holomorphic $k$-dynamical system $(\tilde F_{N,M})$ in the following way:
 for all $M\geq N\in \N^k$ ,$$\tilde F_{N,M}\coloneqq \gamma_M\circ F^{M-N}\circ \gamma_N^{-1}.$$
Define a  holomorphic dynamical system $( g_{h,j})$ in the following way:
for all $j\geq h\in \N$, $$  g_{h,j}\coloneqq \tilde F_{N_h,N_j}.$$

 The dynamical system $( g_{h,j}\colon X\to X)$ admits a relatively compact orbit. Indeed $g_{0,j}(\gamma_0(p))\subset K$ for all $j\in \N$.
Let  $(A, \tilde \alpha_h)$ be the CDL  for  $(g_{h,j})$ given by \cite[Section 2]{A}. 
This extends,  as in Lemma \ref{directlimit}, to  a CDL   $(A, \tilde \alpha_N)$ for $(\tilde F_{N,M})$ satisfying $\tilde \alpha_h=\tilde \alpha_{N_h}$ for all $h\in \N$. Clearly, $(A, \alpha_N\coloneqq \tilde \alpha_N\circ \gamma_N^{-1})$ is a CDL for $(F^{M-N})$.

Notice  that, as $h\to\infty$, the map $\tilde \alpha_{N_h}$ converges uniformly on compacta  to a  retraction $\tilde\alpha\colon X\to X$ of rank $d$ such that $\tilde \alpha(X)=A$.
Since $\tilde \alpha$ has rank $d$ on $A$, there exists a neighborhood $U$ of $A$ such that  ${\rm rk}_w(\tilde \alpha)=d$ for all $w\in U$.
Up to passing to a subsequence, we may assume that the sequence $(\tilde F_{0,N_h}(\gamma_0(p)))_{h\geq 0}$ converges to $\tilde \alpha_0(\gamma_0(p))\in A$ as $h\to \infty$.
There exists  a neighborhood $V\subset \subset U$ of  $\tilde\alpha_0(\gamma_0(p))$ and $\bar h\in \N$ such that ${\rm rk}_w(\tilde \alpha_{N_h})=d$ for all $h\geq \bar h$ and all $w\in V$.

There exists an index $h_0\geq \bar h$ and $r>0$ such that for all $h\geq h_0$, 
the Kobayashi ball $B(\tilde F_{0,N_{h}}(\gamma_0(p)),r)$  is contained in $V$
and thus
\begin{equation}\label{whiskey}
{\rm rk}_w(\tilde \alpha_{N_h})=d,\quad \forall\,w\in B(\tilde F_{0,N_{h}}(\gamma_0(p)),r).
\end{equation}

Since for all $h\geq 1$ the mapping $\ell$ has not rank identically $d$ on  $B(F^{N_h}(p),\frac{1}{h})$, we have that 
for all $h\geq 1$ the mapping   $\alpha_{N_h}$ has not rank identically $d$ on   $B(F^{N_h}(p),\frac{1}{h})$, and thus for all $h\geq 1$
$\tilde \alpha_{N_h}$ has not rank identically $d$ on $B(\tilde F_{0,N_{h}}(\gamma_0(p)),\frac{1}{h})$, which  contradicts (\ref{whiskey}). This proves (1). Statement (2) follows immediately from $\ell\circ F^M=T^M\circ\ell.$

\end{proof}

\begin{proposition}\label{univ}
 Let $X$ be a cocompact  Kobayashi hyperbolic complex manifold, and  let $F\coloneqq(f_1,\dots, f_k)$ be a $k$-uple of commuting holomorphic self-maps of $X$. Assume that $F$ is of automorphism type, and let $(Z,\ell,T)$ be a CSM for $F$. Then for all $r>0$ and all $p\in X$ there exists $N_0\in \N^k$ such that for all $ N\geq N_0\in \N^k$
\begin{enumerate}
\item  the map $\ell$ is univalent on   $B(F^N(p),r)$,
\item the  map $F^M$ is univalent on   $B(F^N(p),r)$ for all $M\in \N^k$.
\end{enumerate}
\end{proposition}

\begin{proof}
The proof is similar to the one of Proposition \ref{iltipo}.
Fix $r>0$ and $p\in X$.
By contradiction, assume that there exists a strictly increasing sequence $(N_h)$ in $\N^k$ with $N_0=0$, converging to $\infty$ and such that for all $h\geq 1$ the mapping $\ell$ is  not  univalent on $B(F^{N_h}(p),r)$.
Let $K$ be a compact subset in $X$ such that $X={\rm aut}(X)\cdot K$. 
 For all $N\in \N$, let  $\gamma_N$ be an automorphism of $X$ such that $\gamma_N(F^{N}(p))\in K$. 
Define for all $M\geq N\in \N^k$ ,$$\tilde F_{N,M}\coloneqq \gamma_M\circ F^{M-N}\circ \gamma_N^{-1}.$$
Define for all $j\geq h\in \N$, $$  g_{h,j}\coloneqq \tilde F_{N_h,N_j}.$$

 The dynamical system $( g_{h,j}\colon X\to X)$ admits a relatively compact orbit. Let  $(A, \tilde \alpha_h)$ be the CDL  for  $(g_{h,j})$ given by \cite[Section 2]{A}. 
This extends,  as in Lemma \ref{directlimit}, to  a CDL   $(A, \tilde \alpha_N)$ for $(\tilde F_{N,M})$ satisfying $\tilde \alpha_h=\tilde \alpha_{N_h}$ for all $h\in \N$. Clearly, $(A, \alpha_N\coloneqq \tilde \alpha_N\circ \gamma_N^{-1})$ is a CDL for $(F^{M-N})$.

Since for all $h\geq 1$ the mapping $\ell$ is  not  univalent on $B(F^{N_h}(p),r)$, we have that 
for all $h\geq 1$ the mapping   $\alpha_{N_h}$ is  not  univalent on   $B(F^{N_h}(p),r)$, and thus
$\tilde \alpha_{N_h}$ is not univalent on $B(K,r)$. 
Since by assumption $F$ is of automorphism type, it follows    that $\tilde \alpha_{N_h}\to \id_X$
 uniformly on compact subsets as $h\to\infty$, which gives a contradiction.
This proves (1). Statement (2) follows immediately from $\ell\circ F^M=T^M\circ\ell.$
\end{proof}

\begin{corollary}\label{iltiposcende}
Let $X$ be a cocompact  Kobayashi hyperbolic complex manifold. Let $$E=(f_1,\dots f_k,g_1,\dots g_p)$$ be a $(k+p)$-uple of commuting holomorphic self-maps of $X$. Let $F\coloneqq(f_1, \dots,f_k)$.
 Then $${\rm type}(E)\leq {\rm type}(F).$$
\end{corollary}
\begin{proof}
 Let $G\coloneqq(g_1,\dots, g_p)$.
 Let $(Z,\ell,(T_F,T_G))$ be a CSM for $E$ and let $(\Lambda,h,\Phi)$ be a CSM for $F$. Since 
$(Z,\ell,T_F)$ is a quasi-model for $F$, by Remark \ref{defgiusta} there exists a morphism $\hat\eta\colon (\Lambda,h,\Phi)\to (Z,\ell,T_F)$. In particular we have $\ell=\eta\circ h.$ 
We claim that there exists a point $x\in X$ such that ${\rm rk}_x \ell={\rm type}(E)$ and ${\rm rk}_x h={\rm type}(F)$, and this clearly yields the result. 
Fix $z\in X$. By Proposition \ref{iltipo}, there exists $U_0\coloneqq (N_0,M_0)\in \N^{k+p}$ such that for all $U\geq U_0$, we have ${\rm rk}_{E^U(z)} \ell={\rm type}(E)$.
Define $$y\coloneqq G^{M_0}(z).$$
By Proposition \ref{iltipo}, there exists $N_1\in \N^k$ such that for all $N\geq N_1$, we have ${\rm rk}_{F^N(y)} h={\rm type}(F)$.
Let $N_2\coloneqq {\rm max} (N_0, N_1).$ Setting $$x\coloneqq  F^{N_2}(G^{M_0}(z)),$$
the claim is proved.
\end{proof}

The remaining of the section is devoted to a useful procedure to construct a CSM for a finite family of commuting holomorphic self-maps.

Let $X$ be a cocompact  Kobayashi hyperbolic complex manifold, and  let $F\coloneqq(f_1,\dots, f_k)$ be a $k$-uple of commuting holomorphic self-maps of $X$. 
Let $(\Lambda,h,\Phi)$ be a  CSM  for $F$. Let $g$ be a holomorphic self map of $X$ commuting with $F$. Since $$\v_j \circ (h \circ g)=(h \circ g)\circ  f_j\quad \mbox{for all}\ 1\leq j\leq k,$$ 
the triple $(\Lambda, h\circ g, \Phi)$ is a  quasi-model for $F$ with Kobayashi hyperbolic base space.
\begin{definition}
We denote by $\Gamma_{F}(g)$ the  holomorphic self-map of $\Lambda$ which  commutes with $\v_j$ for all $1\leq j\leq k$  and which satisfies $$\Gamma_{F}(g)\circ h=h\circ g,$$
 whose existence and uniqueness follows from    Remark \ref{defgiusta}.
\end{definition}

\begin{remark}
Notice that $\Gamma_F(\id_X)=\id_\Lambda$.
\end{remark}

\begin{lemma}\label{gammafunt}
If $g_1,g_2\colon X\to X$ are two holomorphic self-maps commuting with $F$, then $$\Gamma_F(g_1\circ g_2)=\Gamma_F(g_1)\circ \Gamma_F(g_2).$$
\end{lemma}
\begin{proof}
We have that $$\Gamma_F(g_1)\circ \Gamma_F(g_2)\circ h=\Gamma_F(g_1)\circ h\circ g_2=h\circ g_1\circ g_2,$$ and the result follows from the uniqueness of $\Gamma_F(g_1\circ g_2)$.
\end{proof}

\begin{corollary}\label{automorfismo}
If $g$ is an automorphism of $X$, then $\Gamma_{F}(g)$ is an automorphism of $\Lambda$.
\end{corollary}
\begin{proof}
Clearly $\Gamma_F(g^{-1})$ is an holomorphic inverse for $\Gamma_{F}(g)$.
\end{proof}
\begin{corollary}
Let $g_1,g_2\colon X\to X$ be two holomorphic self-maps commuting with $F$. If $g_1$ commutes with $g_2$, then $\Gamma_F(g_1)$ commutes with $\Gamma_F(g_2).$
\end{corollary}
\begin{proof}
 By Lemma \ref{gammafunt}, $$\Gamma_{F}(g_1)\circ \Gamma_{F}(g_2)=\Gamma_{F}(g_1\circ g_2)=\Gamma_{F}(g_2\circ g_1)=\Gamma_{F}(g_2)\circ \Gamma_{F}(g_1).$$
\end{proof}

\begin{theorem}\label{compos}
 Let $X$ be a cocompact  Kobayashi hyperbolic complex manifold. Let $$E=(f_1,\dots f_k,g_1,\dots g_p)$$ be a $(k+p)$-uple of commuting holomorphic self-maps of $X$. Let $F\coloneqq(f_1, \dots,f_k)$  and let $G\coloneqq(g_1,\dots, g_p)$.
Let $(\Lambda, h,\Phi)$ be a  CSM for $F$ and assume that $\Lambda$ is cocompact. Let $(\Omega,r,\Theta)$ be a  CSM for $\Gamma_F(G)$. 
Denote $H\coloneqq \Gamma_F(G)$ and  $\Xi\coloneqq \Gamma_{H} (\Phi)$.
Then a  CSM for  $E$ is given by the triple $(\Omega, r \circ h, (\Xi, \Theta)).$ 
\end{theorem}
\begin{proof}
Denote $\Theta=(\vartheta_1,\dots,\vartheta_p), H=(\eta_1,\dots \eta _p),\Xi=(\xi_1,\dots,\xi_k), \Phi=(\v_1,\dots \v_k).$
First of all, notice that by Corollary \ref{automorfismo}, $\Xi$ is a $k$-uple of  automorphisms of $\Omega$.
Let $(\Delta,s,(A,B))$ be a  semi-model for  $E$ such that $\Delta$ is Kobayashi hyperbolic, where $A=(\alpha_1,\dots \alpha_k)$ and $B=(\beta_1,\dots, \beta_p).$
\SelectTips{eu}{12}
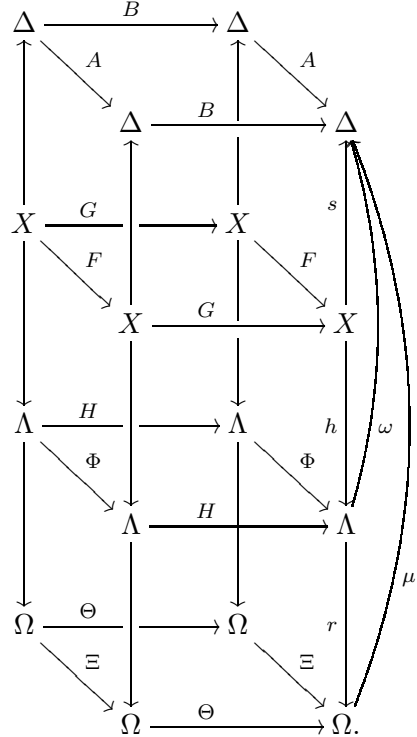
\[ \xymatrix{\Delta  \ar[rd]^A  \ar[rr]^B && \Delta\ar[rd]^A\\
& \Delta   \ar[rr]^(.35)B &&\Delta \\
X\ar[uu]\ar[dd]  \ar[rd]^F  \ar'[r]^(.60)G[rr] && \ar'[u][uu] X\ar'[d][dd]\ar[rd]^F\\
& X \ar[uu] \ar[dd]  \ar[rr]^(.35)G &&X \ar[uu]^(.60)s\ar[dd]_h\\
\Lambda\ar'[r]^(.60){H}[rr]\ar[dd]\ar[rd]^\Phi &&\Lambda\ar[dd]  \ar[rd]^\Phi\\
& \Lambda\ar[dd]\ar[rr]^(.35){H}&&\Lambda\ar[dd]_r \ar@/_1pc/[uuuu]_(.25)\omega\\
\Omega\ar'[r]^(.60)\Theta[rr]\ar[rd]^{\Xi} &&\Omega \ar[rd]^{\Xi}\\
& \Omega\ar[rr]^(.35)\Theta&&\Omega\ar@/_2pc/[uuuuuu]_(.25)\mu.}
\]
The triple $(\Delta,s,A)$ is a quasi-model for $F$. By the universal property stated in Remark \ref{defgiusta}, there exists a
morphism $\hat\omega\colon (\Lambda, h,\Phi)\to (\Delta,s,A)$.

We claim  that  $(\Delta, \omega, B)$ is a quasi-model for $H$, and it suffices to show that 
$$\omega\circ \eta_j=\beta_j\circ \omega,\quad\forall\, 1\leq j\leq p.$$
Fix $1\leq j\leq p$. Then for all $N\in \N^k$,
\begin{align*}
 \omega\circ \eta_j\circ \Phi^{-N}\circ h&=\omega\circ \Phi^{-N} \circ \eta_j\circ h=A^{-N}\circ \omega\circ h\circ g_j=A^{-N}\circ s\circ g_j=\\
 &=A^{-N}\circ \beta_j\circ s=\beta_j\circ A^{-N}\circ \omega\circ h=\beta_j\circ \omega\circ\Phi^{-N}\circ h.
\end{align*}
Hence $\omega\circ \eta_j=\beta_j\circ \omega$ on $\bigcup_{N\in \N^k}\Phi^{-N}\circ h=\Lambda$.

By the universal property stated in  Remark  \ref{defgiusta}, there exists a morphism $\hat\mu\colon (\Omega,r,\Theta)\to (\Delta, \omega, B)$.
We claim that the holomorphic map $\mu\colon \Omega\to \Delta$ also defines a morphism $\hat \mu\colon (\Omega, r \circ h, (\Xi, \Theta))\to (\Delta, s, (A,B)),$ and it suffices to show that 
$$\mu\circ\xi_j=\alpha_j\circ \mu,\quad\forall\, 1\leq j\leq k.$$
Fix $1\leq j\leq k$. For all  $N\in \N^p$,
$$\mu\circ \xi_j\circ \Theta^{-N}\circ r=\mu\circ \Theta^{-N}\circ \xi_j\circ r=\mu\circ \Theta^{-N}\circ r\circ \v_j=B^{-N}\circ \omega\circ \v_j=\alpha_j\circ B^{-N}\circ \omega=\alpha_j\circ \mu\circ \Theta^{-N}\circ r.$$
Hence $\mu\circ\xi_j=\alpha_j\circ \mu$ on $\bigcup_{N\in \N^p} \Theta^{-N}(r(\Lambda))=\Omega.$
\end{proof}

\begin{proposition}\label{divgamma}
Let $X$ be a cocompact  Kobayashi hyperbolic complex manifold.
 Let  $F\coloneqq(f_1,\dots, f_k)$ be a $k$-uple of commuting holomorphic self-maps of $X$. Then 
for all holomorphic self-map $g\colon X\to X$ commuting with $F$ we have $c(\Gamma_F(g))=c(g)$.
\end{proposition}
\begin{proof}
Let $(\Lambda, h,\Phi)$ be a  CSM for $F$, and set $E\coloneqq (f_1,\dots, f_k,g)$. Let $(Z,\ell,(T_F,T_g))$ be a CSM for $E$. Let $$\hat t\colon (\Lambda, h,\Phi)\to(Z,\ell ,T_F)$$ be the morphism given by the universal property stated in Remark \ref{defgiusta}. Then, arguing as in the proof of Theorem \ref{compos}, we see that the triple $(Z,t,T_g)$ is a quasi-model for $\Gamma_F(g)\colon \Lambda\to\Lambda$. But then $$c(g)\geq c(\Gamma_F(g))\geq c(T_g)=c(g),$$
where the last equality follows from Proposition \ref{disdiv}.
\end{proof}

\section{Applications to the unit ball}\label{ball}

We begin this section recalling some definitions and results for the unit ball $\B^q\subset \C^q$.
\begin{definition}\label{jafar}
The Siegel upper half-space $\mathbb H^q$ is defined by $$\mathbb{H}^q=\left\{(z,w)\in \C\times \C^{q-1}, \Im(z)>\|w\|^2\right\}.$$ Recall that $\mathbb H^q$ is biholomorphic to the ball $\B^q$ via the {\sl Cayley transform} $\Psi\colon \B^q\to \H^q$ defined as $$\Psi(z,w)=\left(i\frac{1+z}{1-z},\frac{iw}{1-z}\right), \quad (z,w)\in \C\times \C^{q-1}.$$

Let $\langle\cdot, \cdot\rangle$ denote the standard Hermitian product in $\C^q$. In several complex variables, the natural generalization of the non-tangential limit at the boundary is the following.
 If $\zeta\in \partial\B^q$, then the set $$K(\zeta,R)\coloneqq\{z\in \B^q: |1-\langle z,\zeta\rangle|< R(1-\|z\|)\}$$ is a {\sl Kor\'anyi region} of {\sl vertex} $\zeta$ and {\sl amplitude} $R> 1$.
If $f\colon \B^q \to \C^m$ is a holomorphic map, then we say that $f$ has {\sl $K$-limit} $L\in \C^m$  at
$\zeta$ (we write $K\hbox{-}\lim_{z\to \zeta}f(z)=L$) if for
each sequence $(z_n)$ converging to $\zeta$ such that
$(z_n)$  belongs eventually to some Kor\'anyi region of vertex $\zeta$, we have
that $f(z_n)\to L$.

\end{definition}

\begin{definition}\label{magamago'}
A point $\zeta\in \partial \B^q$ such that $K\hbox{-}\lim_{z\to\zeta}f(z)=\zeta$ and
$$\liminf_{z\to\zeta}\frac{1-\|f(z)\|}{1-\|z\|}=\lambda<+\infty$$ is called a {\sl boundary regular fixed point} (BRFP for short), and $\lambda$ is called  its {\sl dilation}.

\end{definition}

The following  result by Herv\'e \cite{H} generalizes the classical Denjoy--Wolff theorem in the unit disc.
\begin{theorem}
Let $f\colon \B^q\to \B^q$ be holomorphic. Assume that $f$ admits no fixed points in $\B^q$.
Then there exists a  point $p\in \de \B^q$, called the {\sl Denjoy--Wolff point} of $f$, such that $(f^n)$ converges uniformly on compact subsets to the constant map $z\mapsto p$. The Denjoy--Wolff point of $f$ is a BRFP and 
its dilation $\lambda\in(0,1]$.
\end{theorem}

\begin{remark}\label{peterpan}
Let $f\colon\B^q\to\B^q$ be a holomorphic self-map  without fixed points, and let $\lambda$ be the dilation at its Denjoy--Wolff  point. Then by \cite[Proposition 5.8]{AB} the divergence rate of $f$ satisfies $$c(f)=-\log \lambda.$$ 
\end{remark}

\begin{definition}
A holomorphic self-map $f\colon \B^q\to \B^q$ is called
\begin{enumerate}
\item {\sl elliptic} if it admits a fixed point $z\in \B^q$,
\item {\sl parabolic} if it admits no fixed points $z\in \B^q$, and its dilation at the Denjoy--Wolff point is  equal to $1$,
\item {\sl hyperbolic } if it admits no fixed points $z\in \B^q$, and its dilation at the Denjoy--Wolff point  is strictly less than $1$.  
\end{enumerate}
\end{definition}

For a proof of the following result, see, {\sl e.g.}, \cite[Proposition 2.9]{BCMpluri}
\begin{proposition}\label{dilimplica}
Let $f\colon \B^q\to \B^q$ be holomorphic. Assume that $f$ admits a BRFP $p\in\partial\B^q$ with dilation $0<\lambda<1$.
Then $f$ is hyperbolic with Denjoy--Wolff  point $p$.
\end{proposition}

In the unit ball $\B^q$ every holomorphic retract  is an  affine slice, that is,  the intersection of $\B^q$ with an affine complex space (see {\sl e.g.}, \cite[Corollary 2.2.16]{aba1}). In particular, every holomorphic retract is (affine) biholomorphic to a (possibly lower dimensional) ball. Also, recall that a {\sl complex geodesic} for $\B^q$ is a holomorphic map $\v:\D \to \B^q$ such that for all $\zeta, \xi\in \D$ it holds $k_\D(\zeta,\xi)=k_{\B^q}(\v(\zeta), \v(\xi))$. The image of each complex geodesic for $\B^q$ is a one-dimensional affine slice of $\B^q$, and, conversely, every one-dimensional affine slice of $\B^q$ is the image of a  complex geodesic for $\B^q$ (see {\sl e.g.} \cite[Corollary 2.6.9]{aba1}). 

By Theorem \ref{principaleforward}, Proposition \ref{disdiv} and Remark \ref{peterpan} we have the following.

\begin{theorem}\label{forwardpalla}
Le $F=(f_1,\ldots, f_k)$ be a family of commuting holomorphic self-maps of $\B^q$. Then there exists a  CSM $(\B^d, \ell, T\coloneqq(\tau_1,\dots,\tau_k))$ for $F$, where $0\leq d\leq q$. Moreover, if for some $1\leq j\leq k$ the map $f_j$ is hyperbolic with dilation $\lambda_j$ at its Denjoy--Wolff point, then $d\geq 1$ and $\tau_j$ is hyperbolic with dilation $\lambda_j$ at its Denjoy--Wolff point.
\end{theorem}

The following theorem is proved in \cite[Thm. 3.3]{Br}.
\begin{theorem}\label{fil}
Let $f,g\colon \B^q\to \B^q$ be commuting holomorphic self-maps without inner fixed points.
Let $p_f,p_g\in \partial\B^q$ be the respective Denjoy--Wolff points. Assume that $p_f\neq p_g$.
Let $\Delta$ be the complex geodesic whose closure contains $p_f$ and $p_g$.
Then $\Delta$ is invariant for $f$ and $g$ and $f|_\Delta,g|_\Delta$ are commuting hyperbolic automorphisms of $\Delta$.
 \end{theorem}
\begin{remark}\label{reciprocaldilations}
It  follows from  Theorem \ref{fil}, that the map $f$ is hyperbolic with a BRFP in $p_g$ and that  the dilation of $f$ at $p_g$ is reciprocal to its dilation at $p_f$.  A symmetric statement holds for $g$.
\end{remark}

The next results generalize Theorem \ref{fil} to a family of commuting holomorphic maps.

\begin{proposition}\label{commonDW}
Let $A$ be a set of indices. Let $\{f_i\}_{i\in A}$ be a family of commuting holomorphic self-maps of $\B^q$ without inner fixed points. For all $i\in A$ let  $p_i\in \partial \B^q$ be the Denjoy--Wolff point of $f_i$. 
\begin{enumerate}
\item If there exists $i\in A$ such that $f_i$ is parabolic, then there exists $p\in \de\B^q$ such that $p_i=p$ for all $i\in A$.
\item If for all $i\in A$ the map $f_i$ is hyperbolic then
\begin{itemize}
\item[i)] either there exists $p\in \de\B^q$ such that $p_i=p$ for all $i\in A$,
\item[ii)] or there exist $p,p'\in \de \B^q$, $p\neq p'$ such that  $p_i\in \{p, p'\}$ for all $i\in A$,
 and there exist  $j,k\in A$ such that $p_{j}=p$ and $p_{k}= p'$. Moreover, for all $i\in A$,  the complex geodesic $\Delta$ whose closure contains $p$ and $ p'$ is $f_i$-invariant, and $f_i|_\Delta: \Delta\to \Delta$ is a hyperbolic automorphism of $\Delta$ with $p$ and $ p'$ BRFP. In particular, both $p$ and $p'$ are BRFP's for $f_i$ for all $i\in A$.
\end{itemize}
\end{enumerate}
\end{proposition}

\begin{proof}
If $f_{i}$ is parabolic and $p_k\neq p_i$ for some $k\in A$ we immediately reach a contradiction by Remark \ref{reciprocaldilations}.

Suppose now that for all $i\in A$ the map $f_i$ is hyperbolic. Suppose $ p_j\neq  p_k$ for some $j\neq k$. If for all $m\in A$ we have that $ p_m\in \{ p_j, p_k\}$, then the result follows applying Theorem \ref{fil} to the couple of commuting holomorphic self-maps $f_j, f_m$ if $ p_j\neq  p_m$ or to the couple $f_k, f_m$ if $ p_k\neq  p_m$.

Thus, in order to conclude the proof, we need to show that if $ p_j\neq  p_k$ for some $j\neq k$ then for all $m\in A$ it holds
$ p_m\in \{ p_j, p_k\}$. Assume by contradiction that this is not the case, and let $m\in A$ be such that $ p_m\not \in  \{ p_j, p_k\}$. For all $t\in A$, let $\lambda_t\in (0,1)$ be the  dilation  of $f_t$ at $ p_t$. 

By Theorem \ref{fil}, both $ p_k$ and $ p_m$ are BRFP's for $f_j$ with dilation $1/\lambda_j$. Similarly, $ p_j$ and $ p_m$ are BRFP's for $f_k$ with dilation $1/\lambda_k$ and $ p_j$ and $ p_k$ are BRFP's for $f_m$ with dilation $1/\lambda_m$. 

We first assume that $\lambda_j<\lambda_k$.  Let $g:=f_j\circ f_k$. Since $ p_j$ is the Denjoy--Wolff point of $f_j$ and a BRFP for $f_k$, it follows by the chain rule for dilations (see {\sl e.g.} \cite[Lemma 2.6]{abatebracci})  that $ p_j$ is a BRFP for $g$  with dilation given by the product of the dilations of $f_k$ and of $f_j$ at $ p_j$. Therefore, the dilation of $g$ at $ p_j$ is $\mu:=\frac{\lambda_j}{\lambda_k}<1$. By Theorem \ref{dilimplica}, $ p_j$ is the Denjoy--Wolff point of $g$. Now, $g$ commutes with $f_m$, and by Theorem \ref{fil}, it has  a BRFP at $ p_m$ with dilation $1/\mu$. However, since $ p_m$ is a BRFP for both $f_j$ and $f_k$, the dilation of $g=f_j\circ f_k$ at $ p_m$ is given by $\frac{1}{\lambda_j\lambda_k}$. This implies
\[
\frac{\lambda_k}{\lambda_j}=\frac{1}{\lambda_j\lambda_k},
\]
which gives a contradiction since $\lambda_k<1$. 

A similar argument works in case $\lambda_k<\lambda_j$. If $\lambda_j=\lambda_k$, one can argue as before replacing $g$ with $g=f_j^{ 2}\circ f_k$.
\end{proof}

\begin{theorem}\label{fin1}

Let  $F\coloneqq(f_1,\dots, f_k)$ be a $k$-uple of commuting  hyperbolic holomorphic self-maps of $\B^q$.
For all $1\leq j\leq k$ let   $0<\lambda_j<1$ be the dilation of $f_j$ at its Denjoy--Wolff point.
Then there exists a CSM $(\B^d,h,\Theta\coloneqq  (\tau_1,\dots,\tau_k))$ for $F$, where $1\leq d\leq  q$, satisfying the following properties. For all $1\leq j\leq k$ the automorphism $\tau_j$ is  hyperbolic   with fixed points $-e_1, e_1$ and the dilation at its Denjoy--Wolff point is $\lambda_j$. Moreover, one of the following two cases occurs.
 \begin{enumerate}
\item If the $f_j$'s have the same Denjoy--Wolff point $p\in \partial \B^q$,  then $\hbox{K-}\lim_{z\to p}h(z)=e_1$ and $e_1$ is the common Denjoy--Wolff point of the $\tau_j$'s.
\item If the set of Denjoy--Wolff points of the $f_j$'s is formed by two points $p,p'\in\partial \B^q$, then  $\hbox{K-}\lim_{z\to p}h(z)=e_1$, and $\hbox{K-}\lim_{z\to p'}h(z)=-e_1$. The point $e_1$ (respectively $-e_1$) is the Denjoy--Wolff point of $\tau_j$ if and only if $p$ (respect. $p'$) is the Denjoy--Wolff point of $f_j$. 
Let $\Delta$ be the complex geodesic whose closure contains $p$ and $p'$. Then $h|_\Delta: \Delta \to \C e_1\cap \B^d$ is an affine bijective transformation.
 \end{enumerate}
\end{theorem}

\begin{proof}
Let $(\B^d, h, \Theta=(\tau_1,\dots \tau_k))$ be the CSM given by Theorem \ref{forwardpalla}.
By \cite[Proposition 1.8]{GdF}, the mappings $(\tau_1,\dots \tau_k)$ have their two boundary repelling fixed points in common.  Up to conjugacy with an automorphism of $\B^d$  we can assume them to be $-e_1, e_1$. 

Now, by Proposition \ref{commonDW}, there are two cases: either the $f_j$'s have a common Denjoy--Wolff point $p\in \partial \B^q$ or the set of Denjoy--Wolff points if formed by two points $p,p'\in \partial \B^q$.

In the first case,  up to conjugating (if necessary) the family $\Theta$ with the  automorphism $-{\sf id}_{\B^d}$, we can assume that $e_1$ is the Denjoy--Wolff point of $\tau_1$. Now, let $x\in \B^q$. Note that the orbit  $(f_1^{ n}(x))$ converges to $p$ and
\begin{equation}\label{limitDW}
\lim_{n\to \infty}h(f_1^{ n}(x))=\lim_{n\to \infty}\tau_1^{ n}(h(x))= e_1.
\end{equation}
Since $f_1$ is hyperbolic, by \cite[Sect. 3.5]{BP} the orbit $(f_1^{ n}(x))$ is eventually contained in a Kor\'anyi region with vertex in its Denjoy--Wolff point and by \cite[Lemma 5.2]{AB} it satisfies therefore the hypotheses of  \cite[Theorem 5.2]{AB}, hence
$$K\hbox{-}\lim_{z\to p}h(z)=e_1.$$
Moreover, by the same token, for all $2\leq j\leq k$ the orbit $(f_j^{ n}(x))$  lies eventually in a Kor\'anyi region with vertex $p$, therefore
$h(f_j^{ n}(x))=\tau_j^{ n}(h(x))\to e_1$, proving that $e_1$ is the Denjoy--Wolff point of $\tau_j$ for all $1\leq j\leq k$.

In the second case, we conjugate the family $F$ with an automorphism of $\B^q$ in such a way that $p=e_1$ and $p'=-e_1$. Then $\Delta=\D\times \{0\}\subset \C\times \C^{q-1}$. 
Let $1\leq m\leq k$ such that the Denjoy--Wolff point of $f_m$ is $e_1$.  Then up to conjugating (if necessary) the family $\Theta$ with the  automorphism $-{\sf id}_{\B^d}$, we can assume that the Denjoy--Wolff point of $\tau_m$ is $e_1$.
Denote $$h(\zeta,0,\dots ,0)\coloneqq (h_1(\zeta), \dots , h_d(\zeta)), \quad \zeta\in \D.$$
By  the proofs of \cite[Proposition 1.3 and 2.7]{GdF} it follows that $h_1\colon \D\to \D$ is a hyperbolic automorphism with fixed points $\{1,-1\}$, while $h_i(\zeta)\equiv 0$ for all $2\leq i\leq d$.
This implies that  $K\hbox{-}\lim_{z\to e_1}h(z)=e_1$ and  $K\hbox{-}\lim_{z\to -e_1}h(z)=-e_1$. Arguing as before it is easy to see
 that for all $1\leq j\leq k$ such that $f_j$ has Denjoy--Wolff point $e_1$ (resp. $-e_1$), the automorphism $\tau_j$ has Denjoy--Wolff point $e_1$ (resp $-e_1$).
\end{proof}

\begin{corollary} Let  $F\coloneqq(f_1,\dots, f_k)$ be a $k$-uple of commuting  hyperbolic holomorphic self-maps of $\B^q$. 
For all $1\leq j\leq k$ let   $0<\lambda_j<1$ be the dilation of $f_j$ at its Denjoy--Wolff point.
Then there exists a CSM $(\H^d,\ell,\Phi\coloneqq (\varphi_1,\dots,\varphi_k))$ for $F$ with $1\leq d\leq  q$,   such that one of the following two cases occurs.
\begin{enumerate}
\item If the  $f_j$'s share the same Denjoy--Wolff  point $p\in \partial \B^q$, then  $K\hbox{-}\lim_{z\to p}\ell(z)=\infty$, and  each $\varphi_j$ has the form 
\begin{equation}\label{prima}
(z,w)\mapsto \left(\frac{1}{\lambda_j} z,\frac{e^{it^j_1}}{\sqrt \lambda_j}w_1,\dots, \frac{e^{it^j_{d-1}}}{\sqrt \lambda_j}w_{d-1} \right),
\end{equation}
where $t^j_a\in \R$ for $1\leq a\leq d-1$.
\item If the set of Denjoy--Wolff points of the $f_j$'s  consists of two points $p, p'\in \partial \B^q$, then  
\begin{itemize}
\item[i)] if $p$ is the Denjoy--Wolff point of $f_j$, then $K\hbox{-}\lim_{z\to p}\ell(z)=\infty$ and
 $\varphi_j$ has  the form (\ref{prima})
\item[ii)] if $p'$ is the Denjoy--Wolff point of $f_j$, then $K\hbox{-}\lim_{z\to p'}\ell(z)=0$ and   $\varphi_j$ has  the form
\begin{equation}\label{seconda}
(z,w)\mapsto \left(\lambda_j z,e^{it^j_1}\sqrt \lambda_jw_1,\dots, e^{it^j_{d-1}}\sqrt \lambda_jw_{d-1} \right),
\end{equation}
where $t^j_a\in \R$ for $1\leq a\leq d-1$. Let $\Delta$ be the complex geodesic whose closure contains $p$ and $p'$. Then $\ell|_\Delta: \Delta \to\H$ is a M\"obius transformation.
\end{itemize}
\end{enumerate}
\end{corollary}

\begin{proof}
Let $(\B^d, h, \Theta)$ be the CSM for $F$ given by Theorem \ref{fin1}.
 Let $C:\B^d \to \H^d$ be  the Cayley transform from $\B^d$ to $\H^d$ which maps $e_1$ to $\infty$ and $-e_1$ to $0$. Setting  $\tilde{\ell}:=C \circ h$ and $\tilde{\varphi}_j:= C\circ \tau_j\circ  C^{-1}$, it follows that $(\H^d, \tilde{\ell}, (\tilde{\varphi}_1,\dots ,\tilde{\varphi}_k))$ is a CSM for $F$.

In the first case, note that for each $j$, the point $\infty\in \partial \H^d$ is  the Denjoy--Wolff point of $\tilde{\varphi}_j$, and the point $0$ is a BRFP. By \cite[Proposition 2.2.11]{aba1}, for all $1\leq j\leq k$ there exists a unitary $(d-1)\times (d-1)$-matrix $U_j$ such that $\tilde{\varphi}_j$  has the form
$$(z,w)\mapsto\left (\frac{1}{\lambda_j}z,\frac{1}{\sqrt \lambda_j}U_jw\right).$$
Since the unitary matrices $U_1,\dots ,U_k$ commute, there exists a unitary matrix $Q$ such that for all $1\leq j\leq k$ we have
$$Q U_j Q^{-1} = D_j,$$ where $D_j$ is a diagonal unitary matrix $$D_j\coloneqq {\sf diag}(e^{it^j_1},\dots, e^{it^j_{d-1}} ),$$ with $t^j_a\in \R$ for all $1\leq a\leq d-1.$
Define an automorphism of $\H^d$ by setting $\Psi(z,w)=(z,Qw).$ 
For all $1\leq j\leq k$ set $\v_j\coloneqq  \Psi\circ \tilde{\varphi}_j\circ \Psi^{-1}$, and set $\Phi\coloneqq (\v_1,\dots \v_k)$,  $\ell \coloneqq \Psi\circ \tilde{\ell}$.    Then $(\H^d, \ell, \Phi)$ is a CSM for $F$ satisfying the required properties. 
The second case is similar.

\end{proof}

In the remaining of the section we study obstructions to commutation in the unit ball.
Cowen in \cite[Corollary 4.1]{Co2} proved that if $f,g:\D \to \D$ are commuting holomorphic maps different from the identity, then $f$ and $g$ are either both elliptic, or  both hyperbolic, or both parabolic. In higher dimensions the situation is far more complicated. Examples of elliptic maps different from the identity commuting with hyperbolic or parabolic maps can be easily constructed in $\B^q$ for $q\geq 2$. On the other hand we can say something about such elliptic self-maps.
\begin{remark} Let $q\geq2 $.
Let $f\colon \B^q\to \B^q$ be an elliptic self-map and assume that it commutes with a  parabolic or hyperbolic self-map $g\colon \B^q\to \B^q$. Then ${\sf Fix}(f)$ is an affine slice of $\B^q$ of nonzero dimension, and the Denjoy--Wolff point of $g$  is in $ \overline{\hbox{\sf Fix}(g)}$. 
\end{remark}

The next natural question is if a hyperbolic mapping can commute with  a parabolic mapping in $\B^q$. If $q=1$ this is not possible by Cowen's Theorem.  Notice that in dimension one every hyperbolic map is of automorphism type. The following result generalizes Cowen's theorem to several variables assuming that the hyperbolic map is of automorphism type.
\begin{theorem}\label{teo-aut}
Let $f: \B^q\to \B^q$ be a  hyperbolic holomorphic self-map of automorphism type. If $g: \B^q\to \B^q$ is a holomorphic self-map commuting with $f$, then $g$ is not  parabolic.
\end{theorem}
\begin{proof}

By \cite[Theorem 4.6]{A}, $f$ admits a CSM 
$(\H^q,h,\v)$, where  $\v\colon \H^q\to \H^q$ is  a hyperbolic automorphism of the form
\begin{equation}\label{formauno}
\v(z,w)= \left(\frac{1}{\lambda} z,\frac{e^{it_1}}{\sqrt \lambda}w_1,\dots, \frac{e^{it_{q-1}}}{\sqrt \lambda}w_{q-1} \right).
\end{equation}
The map $g$ induces a holomorphic self-map $\Gamma_{f}(g)\colon \H^q\to\H^q$ commuting with $\v$ and satisfying 
\begin{equation}\label{pluto}
h\circ g=\Gamma_{f}(g)\circ h.
\end{equation} By \cite[Propositions 1.3 and 1.7]{GdF}, the mapping  $\Gamma_{f}(g)$ leaves the one-dimensional slice  $\Delta\coloneqq \{w=0\}\subset \H^q$ invariant and  there exists $a>0$ such that 
$$\Gamma_{f}(g)(z,0)=(az,0), \quad \forall z\in \H.$$  
If $a\neq 1$, then $\Gamma_{f}(g)$ is  hyperbolic, and hence  
$g$ is also  hyperbolic.
Assume now that $a=1$, that is, $\Gamma_{f}(g)$ fixes every point of the slice $\Delta$. 
Since $\H^q=\bigcup_{n\geq 0}\v^{-n}h(\B^q)$, it follows that $h^{-1}(\Delta)\neq \varnothing$. Notice that $h^{-1}(\Delta)$ is $f$-invariant.
Let $x\in h^{-1}(\Delta)$. 
Let $r>k_{\B^q}(x,g(x))$. By Proposition \ref{univ}, there exists $n_0=n_0(x,r)$ such that for all $n\geq n_0$ the mapping $h$ is univalent on the Kobayashi ball $B(f^n(x),r)$.
We have that 
$$k_{\B^q}(g(f^n(x)),f^n(x))=k_{\B^q}(f^n(g(x)),f^n(x))\leq k_{\B^q}(g(x),x)<r,$$
thus $g(f^n(x))\in B(f^n(x),r)$.
Since $f^n(x)\in h^{-1}(\Delta)$, it follows by (\ref{pluto}) that $h(g(f^n(x)))=h(f^n(x))$, which implies that $g(f^n(x))=f^n(x)$.
Hence $g$ is elliptic.
\end{proof}

Hence if the type of the hyperbolic mapping is maximal, then it cannot commute with a parabolic mapping.
Other obstructions to commutation can be found if the type of the parabolic mapping is too small, as the following results show.

\begin{proposition}\label{para-0}
Let  $f: \B^q\to \B^q$ be a  hyperbolic holomorphic self-map and let $g\colon \B^q\to \B^q$ be a parabolic holomorphic self-map of type $0$. Then $f$ and $g$ do not commute.
\end{proposition}
\begin{proof}
Assume by contradiction that $f$ and $g$ commute.
Let $(Z,\ell,T)$ be a  CSM for the pair $(f,g)$. Since $f$ is hyperbolic we have by Corollary \ref{dimensionemaggioreuno} that ${\rm dim}\, Z\geq1$. Since $g$ has type $0$, it follows from Corollary \ref{iltiposcende} that ${\rm dim}\, Z=0$, contradiction. 
\end{proof}

\begin{proposition}\label{para-1}
Let  $f: \B^q\to \B^q$ be a  hyperbolic holomorphic self-map and let $g\colon \B^q\to \B^q$ be a parabolic holomorphic self-map of type $1$. If $f$ and $g$ commute, then $g$ admits a CSM of the form $(\D,h,\id)$.
\end{proposition}
\begin{proof}
Since $g$ is of type $1$, it admits a CSM  of the form $(\D,h,\v)$, where $\v$ is not hyperbolic. Let $\Gamma_g(f)\colon \D\to \D$ be the holomorphic self-map induced by $f$. By Proposition \ref{divgamma} the mapping $\Gamma_g(f)$ is hyperbolic. Since $\Gamma_g(f)$  commutes with $\v$, it follows from Cowen's theorem \cite[Corollary 4.1]{Co2} (or Theorem \ref{teo-aut} for $q=1$) that  $\v=\id_\H$. 
\end{proof}
\begin{corollary}\label{blu}
Let  $f: \B^q\to \B^q$ be a  hyperbolic holomorphic self-map and let $g\colon \B^q\to \B^q$ be a parabolic holomorphic self-map of type $1$ such that there exists $x\in \B^q$ with $s_1^g(x)>0$. Then $f$ and $g$ do not commute.
\end{corollary}
\begin{remark}
It is actually an open question whether there exists a parabolic self map $g\colon \B^q\to \B^q$ with  ${\rm type}\, g\geq 1$ such that there exists $x\in \B^q$ with $s_1^g(x)=0$.
In other words, it is an open question whether there exists a parabolic self map $g\colon \B^q\to \B^q$ whose  CSM $(\B^k,h,\v)$ is such that $k\geq 1$ and $\v$ is elliptic.  
\end{remark}

The following proposition shows that there cannot be more type-related obstructions to the commutation of a hyperbolic and a parabolic mapping. 
\begin{proposition}\label{example}
Let $m\in \N$, $m\geq 2$. Let $1\leq q\leq m-1$, and let $2\leq p\leq m$.
Then there exist commuting holomorphic self-maps $f,g\colon \H^m\to \H^m$ such that 
$f$ is hyperbolic of type $q$ and $g$ is parabolic of type $p$.
\end{proposition}
\begin{proof}
Let $r\in \R,$ $r\neq 0$. 
Denote a point in $\C^{m+2}$ by $(z,w,y)$, where $z,w\in \C$ and $y\in \C^m$. With a slight abuse of notation, if $m=2$, the triple $(z,w,y)$ simply denotes the point $(z,w)$.
Consider the  holomorphic self maps $ f,g $ of $\H^{m}$ defined by
$$ f(z,w,y)\coloneqq (2z+iw^2, w,\sqrt 2y_1,\dots,\sqrt 2y_{q-1},0,\dots,0)$$
$$   g(z,w,y)\coloneqq (z+ir^2-2rw, w-ir,y_1, \dots, y_{p-2},0,\dots,0).$$
It is easy to see that $f$ is hyperbolic and $g$ is parabolic and that they commute.
We claim that the type of $f$ is $q$, the type of $g$ is $p$, and that the type of the pair $(f,g)$ is ${\rm min}(p-1,q)$.
    
 To prove that the type of $f$ is $q$, we will study the univalent self-map $\varphi\colon \H^{q+1}\to\H^{q+1}$ defined by
$$\varphi(z,w,y_1,\dots ,y_{q-1})\coloneqq (2z+iw^2, w,\sqrt 2y_1,\dots,\sqrt 2y_{q-1}),$$ which clearly has the same type as $f$.
Notice that $\v$ actually defines a univalent mapping on the whole $\C^{q+1}$.   Set
$$\Omega\coloneqq \{ x\in \C^{q+1} \colon \varphi^n(x)\in \H^{q+1}\ \mbox{eventually}\},$$
then $(\Omega, \iota\colon \H^{q+1}\to \Omega,  \varphi)$ is a model for $\varphi$.
For all $n\geq 1$ and $(z,w,y_1,\dots ,y_{q-1})\in \C^{q+1}$, 
\begin{align*}
\varphi^n(z,w,y_1,\dots, y_{q-1})&=(2^nz+(2^n-1)iw^2,w,2^{\frac{n}{2}}y_1,\dots,2^{\frac{n}{2}}y_{q-1} )\\
&=(2^n(z+iw^2)-iw^2,w,2^{\frac{n}{2}}y_1,\dots,2^{\frac{n}{2}}y_{q-1}).
\end{align*}
The point $\varphi^n(z,w,y_1,\dots, y_{q-1})$ is in $\H^{q+1}$ if and only if
$$2^n\left(\Im(z+iw^2)-\sum_{j=1}^{q-1} |y_j|^2\right)>|w|^2+\Im(iw^2),$$ and thus 
 $$\Omega=\{(z,w,y_1,\dots ,y_{q-1})\in\C^{q+1}\colon \Im(z+iw^2)>\sum_{j=1}^{q-1} |y_j|^2\}.$$
The domain $\Omega$ is biholomorphic to $\H^{q}\times \C$. Indeed  the mapping $\eta\colon\Omega\to \H^{q}\times \C$  defined by
$$\eta(z,w,y_1,\dots, y_{q-1})\coloneqq((z+iw^2,y_1,\dots, y_{q-1}),w),$$
is a biholomorphism with inverse
 $$\eta^{-1}((\zeta,\xi_1,\dots ,\xi_{q-1}),\gamma)=(\zeta-i\gamma^2,\gamma,\xi_1,\dots ,\xi_{q-1}).$$ 
 
Define $\ell\colon \H^{q+1} \to\H^{q}$ by $\ell(z,w,y)\coloneqq (z+iw^2,y_1,\dots ,y_{q-1})$. 
Denote by $\Phi\colon \H^{q}\to\H^{q}$ the hyperbolic automorphism defined by $\Phi(\zeta,\xi)=(2\zeta,\sqrt 2 \xi)$, where 
$\zeta\in \C$ and $\xi \in \C^{q-1}$.
Notice that $$\ell(\v(z,w,y_1,\dots, y_{q-1}))=(2z+2iw^2,\sqrt 2 y_1,\dots, \sqrt 2 y_{q-1})=\Phi(\ell(z,w,y_1,\dots, y_{q-1})).$$
Then a CSM for $\v$ is given by $(\H^{q},\ell,\Phi)$, and this proves that ${\rm type\,}f={\rm type\,}\v=q$.

Let now $\pi\colon \H^m\to \H^p$ be given by $\pi(z,w,y)=(z,w, y_1,\dots, y_{p-2})$, and let $\Psi\colon \H^p\to \H^p$ be the parabolic automorphism defined by  $\Psi(a,b,c)\coloneqq (a+ir^2-2rb, b-ir,c),$  where 
$a,b\in \C$ and $c \in \C^{p-2}$. A CSM for $g$ is easily seen to be $(\H^p,\pi,\Psi)$, and thus ${\rm type\,}g=p$. Set $u\coloneqq {\rm min}(p-2,q-1)$.

Then $$\Gamma_g(f)(a,b,c)=(2a+ib^2,b, \sqrt 2 c_1,\dots , \sqrt 2 c_u,0,\dots, 0).$$ Applying the previous argument  we see that the  ${\rm type\,}\Gamma_g(f)=u+1$  and thus  Theorem \ref{compos} yields  ${\rm type\,}(f,g)=u+1$.

\end{proof}
\begin{remark}
Let $r\in \R$, $r\neq 0$. The commuting holomorphic self-maps of $\H^2$ defined by 
$$ f(z,w)\coloneqq (2z+iw^2, w), \quad 
 g(z,w)\coloneqq (z+ir^2-2rw, w-ir),$$
(which correspond to  the minimal dimension case $m=2$ in Proposition \ref{example}) were studied
in \cite{O} to show the existence of a holomorphic self-map of $\B^2$ with  a non-isolated boundary repelling fixed point.
\end{remark}

\begin{remark}
It follows from Proposition \ref{example}  that    ${\rm type}\, (f,g)$ can be strictly smaller than ${\rm min}\{{\rm type}(f),{\rm type}(g)\}.$
\end{remark}


\begin{thebibliography}{99}
\bibitem{aba1} M. Abate, {\sl Iteration theory of holomorphic maps on taut manifolds}, {Research and Lecture Notes in Mathematics. Complex Analysis and Geometry}, {Mediterranean Press, Rende}, {1989}.
\bibitem{abatebracci} M. Abate, F. Bracci, {\sl Common boundary regular fixed points for holomorphic semigroups in strongly convex domains}, Contemp. Math., {\bf 667} (2016), 1--14.
\bibitem{A} L. Arosio, {\sl Canonical models for the forward and backward iteration of holomorphic maps}, J. Geom. Anal. {\bf 27} (2017), no.2, 1178--1210.
\bibitem{AB} L. Arosio, F. Bracci, {\sl  Canonical models for holomorphic iteration}, Trans. Amer. Math. Soc. {\bf 368} (2016), no.5,  3305--3339.
\bibitem{BaPo} I. N. Baker, C. Pommerenke, {\sl On the iteration of analytic functions in a half-plane II}, J. Lond. Math. Soc. (2) {\bf 20} (1979),
255--258.
\bibitem{Bay} F. Bayart, {\sl The linear fractional model on the ball}, Rev. Mat. Iberoam. {\bf 24} (2008), no. 3, 765--824.
\bibitem{Bay2} F. Bayart, {\sl Parabolic composition operators on the ball}, Adv. Math. {\bf 223} (2010), no. 5, 1666--1705.
\bibitem{Bay3} F. Bayart, S. Charpentier, {\sl Hyperbolic composition operators on the ball}, Trans. Amer. Math. Soc.
{\bf 365} (2013), no. 2, 911--938.
\bibitem{Behan} D. F. Behan, {\sl Commuting analytic functions without fixed points}, Proc. Amer. Math. Soc. {\bf 37} (1973), 114--120.
\bibitem{cinzia1} C. Bisi, F. Bracci, {\sl Linear fractional maps of the unit ball: a geometric study}, Adv. Math. {\bf 167} (2002), no. 2, 265--287.
\bibitem{BiGe} C. Bisi, G. Gentili, {\sl Commuting holomorphic maps and linear fractional models},
Complex Variables Theory Appl. {\bf 45} (2001), 47--71.
\bibitem{cinzia2} C. Bisi, G. Gentili, {\sl Schr\"oder equation in several variables and composition operators}, Atti Accad. Naz. Lincei Rend. Lincei Mat. Appl. {\bf 17} (2006), no. 2, 125--134. 
\bibitem{BS} P. S. Bourdon, J. H. Shapiro, {\sl Cyclic Phenomena for
Composition Operators}, Memoirs of Amer. Math. Soc.  125, (596),
1997.
\bibitem{Br} F. Bracci, {\sl Common fixed points of commuting holomorphic maps in the unit ball of $\C^n$}, Proc. Amer. Math. Soc. {\bf 127} (1999), 1133--1141.
\bibitem{BCMpluri} F. Bracci, M. Contreras, S. D\'iaz-Madrigal, {\sl Pluripotential theory, semigroups and boundary behavior of infinitesimal generators in strongly convex domains}, J. Eur. Math. Soc. {\bf 12} (2010), no. 1, 143--154.
\bibitem{BrGe} F. Bracci, G. Gentili, {\sl Solving the Schroeder equation at the boundary in several variables}, Michigan Math. J. {\bf 53}  (2005), no. 2, 337--356.
\bibitem{BP} F. Bracci, P. Poggi-Corradini, {\sl On Valiron's Theorem},
Future Trends in Geometric Function Theory. RNC Workshop
Jyv\"askyl\"a, Rep. Univ. Jyv\"askyl\"a Dept. Math. Stat. \textbf{92}
(2003), 39--55.
\bibitem{BFG} F. Bracci, G. Gentili, P. Poggi-Corradini, {\sl Valiron's construction in higher dimension},  Rev. Mat. Iberoam.  {\bf 26}  (2010), no. 1, 57--76.
\bibitem{Co1} C. C. Cowen, {\sl Iteration and the solution of functional equations for functions analytic in the unit disc}, Trans. Amer. Math. Soc. {\bf 265} (1981), 69--95.
\bibitem{Co2} C.C. Cowen,  {\sl Commuting analytic functions}, Trans. Amer. Math.
Soc. {\bf 283} (1984), 685--695.
\bibitem{CM} C. C. Cowen, B. D. MacCluer, {\sl Schroeder's equation in several variables}, Taiwanese J. Math. {\bf 7} (2003),
129--154.
\bibitem{deFa}  C. de Fabritiis, {\sl Commuting holomorphic functions and hyperbolic automorphisms}, Proc. Amer. Math. Soc. {\bf 124} (1996), 3027--3037.
\bibitem{GdF} C. de Fabritiis, G. Gentili, {\sl On holomorphic maps which commute with hyperbolic automorphisms},  Adv. Math. {\bf 144} (1999), no. 2, 119--136. 
\bibitem{FS} J. E. Forn\ae ss, N. Sibony,
{\sl Increasing sequences of complex manifolds}, Math. Ann.
{\bf 255} (1981), no. 3, 351--360.
\bibitem{He} M. H. Heins, {\sl A generalization of the Aumann-Carath\'eodory ``Starrheitssatz''}, Duke Math. {\bf 8}
(1941), 312--316.
\bibitem{H} M. Herv\'e, {\sl Quelques propri\'et\'es des applications analytiques d'une boule \`a m dimensions dans elle-m\^eme}, J. Math. Pures Appl. \textbf{42} (1963), 117--147.
\bibitem{KRS} V. Khatskevich, S. Reich, D. Shoikhet, {\sl
Abel-Schr\"oder equations for linear fractional mappings and the
Koenigs embedding problem}, Acta Sci. Math. (Szeged) {\bf 69} (2003),
67--98.
\bibitem{Mac} B. D. MacCluer, {\sl Commuting analytic self-maps of the ball}, Pacific J. Math. {\bf 194} (2000), 413--426.
\bibitem{O} O.  Ostapyuk, {\sl Backward iteration in the unit ball}, Illinois J. Math. {\bf 55} (2011), no. 4, 1569--1602.
\bibitem{Po} C. Pommerenke, {\sl On the iteration of analytic
functions in a half-plane I}, J. London Math. Soc. (2) {\bf 19} (1979),
439--447.
\bibitem{Sc} E. Schr\"oder, {\sl \"Uber iterierte Funktionen}, Math.
Ann. {\bf 3} (1871), 296--322.
\bibitem{Shields} A. L. Shields, {\sl On fixed points of
commuting analytic functions}, Proc. Amer. Math. Soc. {\bf 15} (1964), 703--706.
\bibitem{va2} G. Valiron, {\sl Sur l'it\'eration des fonctions holomorphes dans un demi-plan},
Bull. Sci. Math. {\bf 47} (1931), 105--128.

\end{thebibliography}
\end{document}